\documentclass[11pt,draft,a4paper]{amsart}
\usepackage[english]{babel}
\usepackage{amssymb}
\usepackage{graphicx}
\usepackage{geometry}
\usepackage{hyperref}  
\hypersetup{colorlinks=true, linkcolor=blue, anchorcolor=blue, 
citecolor=red, filecolor=blue, menucolor=blue,
urlcolor=blue}

\numberwithin{equation}{section}
\newtheorem{theorem}{Theorem}[section]

\newtheorem{proposition}[theorem]{Proposition}
\newtheorem{lemma}[theorem]{Lemma}
\theoremstyle{remark}
\newtheorem{remark}{Remark}[section]

\theoremstyle{definition}

\def\XXint#1#2#3{{\setbox0=\hbox{$#1{#2#3}{\int}$ }
\vcenter{\hbox{$#2#3$ }}\kern-.58\wd0}}

\begin{document}
\title
[Variable coefficients Schr\"odinger and equation]
{Smoothing estimates for variable coefficients Schr\"odinger equation with electromagnetic potentials}
\begin{abstract}
In this paper we develop the classical multiplier technique to build up virial identities for the electromagnetic variable coefficients Schr\"dinger equation. Following the strategy of \cite{fanvega} we shall use such identities to prove smoothing estimates for the associated flow in a perturbative setting.
\end{abstract}
\date{\today}    
\author{Federico Cacciafesta}
\address{Federico Cacciafesta: 
SAPIENZA --- Universit\`a di Roma,
Dipartimento di Matematica, 
Piazzale A.~Moro 2, I-00185 Roma, Italy}
\email{cacciafe@mat.uniroma1.it}










\maketitle


\section{Introduction and main results}

We consider the electromagnetic Schr\"odinger equation with variable coefficients
\begin{equation}\label{eq1}
\begin{cases}
iu_t(t,x)=Hu(t,x)\\
u(0,x)=f(x),
\end{cases}
\end{equation}
where $u:\mathbb{R}^{1+n}\rightarrow \mathbb{C}$ and the Hamiltonian $H$ is defined as
$$
Hu=-\left(\partial^b_j\left(a_{jk}(x)\partial^b_k u\right)\right)+V(x)u.
$$
Here the covariant derivatives 
$\partial^{b}=(\partial^{b}_{1},\dots,\partial^{b}_{n})$ associated
to the (divergence-free) magnetic potential $ b=(b^1,...,b^n):\mathbb{R}^n\rightarrow\mathbb{R}^n$ are given by 
$$\partial^b_k=\frac{\partial}{\partial x_k}+ib_k(x),$$
while $a=a(x)=[a_{jk}(x)]_{1,n}$ is a symmetric matrix of real
valued functions satisfying
\begin{equation}\label{strutta}
C^{-1}{\rm Id}\leq a(x)\leq C\:{\rm{Id}},\qquad x\in\mathbb{R}^n
\end{equation} 
for some $C>0$. Using the summation convention over
repeated indices, we associate to the matrix $a(x)$ the bilinear form
\begin{equation}\label{form1}
a(v,w)=a_{jk}(x)\overline{v}_jz_k
\end{equation}
and the operator
$$
A\psi=\partial_j(a_{jk}(x)\partial_k\psi).
$$
A well-known phenomenon for 
the Schr\"odinger equation is the gain of regularity of the solutions
with respect to the initial data: indeed, the free flow satisfies
the following classical estimate
\begin{equation}\label{smot1}
\sup_{R>0}\frac1R\int_0^{+\infty}\int_{|x|\leq R}\big|\nabla e^{it\Delta f}\big|^2\leq \|f\|_{\dot H^\frac12}
\end{equation}
(see \cite{const}, \cite{sjolin}, \cite{vega}). In recent years the case of potential perturbations or variable coefficients equations has attracted increasing interest, and several generalizions of estimate \eqref{smot1} have been investigated (see \cite{burq1}, \cite{burq2} \cite{danfan}, \cite{erdgold}, \cite{danfanve} \cite{mizut}, \cite{mizut2}, 
\cite{DanconaFanelli06-a}
and references therein).

The aim of this paper is to develop the multiplier method 
in order to prove virial identities for equation \eqref{eq1}, and to use such identities to prove smoothing estimates  for the electromagnetic Schr\"odinger flow in dimension $n=3$ with variable coefficients, in analogy with the results in \cite{barruiz1}, \cite{barruiz2}, at least in a perturbative setting. Although the smoothing estimates for the Schr\"odinger equation with variable coefficients are not entirely new (see \cite{mizut2}), to the best of our knowledge
our virial identities are original and may be interesting in view of
the applications to more general dispersive estimates (see Remark \ref{ree}).

In order to ensure self-adjointness for the Hamiltonian $H$, we will make the following abstract assumptions:
\begin{itemize}
\item
The principal part $H_{0}u=-\left(\partial^b_j\left(a_{jk}(x)\partial^b_k u\right)\right)$
of the
Hamiltonian $H$ is essentially self-adjoint on $L^2(\mathbb{R}^n)$ with form domain
\begin{equation*}
D(H_{0})=\left\{f:f\in L^2,\int
a(\nabla_bf,\nabla_bf)<\infty\right\}.
\end{equation*}
\item The potential $V$ is a perturbation of $H_{0}$ in the Kato-Rellich sense, i.e. there exists an $\varepsilon>0$ such that
\begin{equation*}
\|Vf\|_{L^2}\leq(1-\varepsilon)\|Hf\|_{L^2}+C\|f\|_{L^2},
\end{equation*}
for all $f\in D(H)$.
\end{itemize}
These assumptions, together with \eqref{strutta}, allow us to define via spectral theorem the linear propagator in a standard way; moreover the perturbed norms
\begin{equation}
\|f\|_{\dot{\mathcal{H}}^s}=\|H^\frac s2f\|_{L^2}
\end{equation}
are conserved by the respective flows, and so they yield the conservation laws
\begin{equation*}
\|e^{itH}f\|_{\dot{\mathcal{H}}^s}=\|f\|_{\dot{\mathcal{H}}^s},\qquad s\geq 0.
\end{equation*}

The first result we prove is a general virial identity for the variable coefficients Schr\"odinger equation. In order to state it we introduce
the following notations, where we use implicit summation over
repeated indices:
\begin{itemize}
\item
$\displaystyle\hat{x}=\frac{x}{|x|},\qquad\langle x\rangle=\sqrt{1+|x|^2}$\\
\item
$V_r^a=a(\hat{x},\nabla V),\quad B^a_\tau=a(\hat{x},B),\quad D^2_a\phi=a_{kj}a_{lm}\partial_j\partial_m\phi$.
\end{itemize}
\begin{theorem}\label{virial}
Let $\phi:\mathbb{R}^n\rightarrow\mathbb{R}$ be a radial, real valued multiplier, and let
\begin{equation}\label{teta}
\Theta_S(t)=\int_{\mathbb{R}^n}\phi|u|^2dx.
\end{equation}
Then the solution $u$ of \eqref{eq1} with initial condition
 $f\in L^2$, $Hf\in L^2$ 
 satisfies the following virial-type identity:
\begin{equation}\label{virsch}
\ddot{\Theta}_S(t)=4\int_{\mathbb{R}^n}\nabla_buD^2_a\phi\overline{\nabla_bu}dx-\int_{\mathbb{R}^n}|u|^2A^2\phi dx-
2\int_{\mathbb{R}^n}\phi'V_r^a|u|^2dx
\end{equation}
$$
+4\mathcal{I}\int_{\mathbb{R}^n}u\phi'a(\nabla_b u,B_\tau^a)dx
$$
$$
+2\int_{\mathbb{R}^n}\big[2a\left(\nabla_bu,\nabla a(\nabla\phi,\nabla_bu)\right)-
a\left(\nabla\phi,\nabla a(\nabla_bu,\nabla_bu)\right)\big]dx.
$$
\end{theorem}
\begin{remark}
Notice that the quantities appearing in \eqref{virsch} are natural generalizations of the analogous quantities that appear in
the flat case, i.e.,
the standard Jacobian $D^2\phi$, the radial component of the electric potential (which in the
unperturbed case is simply $V_r=\hat{x}\cdot\nabla V$)
and the tangential vector field $B_\tau=\hat{x}B$. 
On the other hand, the last two terms containing derivatives of the coefficients are new and do not appear in the classical computation.
\end{remark}

We can prove a similar result for the electromagnetic, variable coefficients wave equation, i.e. 
\begin{equation}\label{eq2}
\begin{cases}
u_{tt}(t,x)=Hu(t,x)\\
u(0,x)=f(x)\\
u_t(0,x)=g(x)
\end{cases}
\end{equation}
with the same notations as before. We can prove the following

\begin{theorem}\label{virial2}
Let $\phi,\psi:\mathbb{R}^n\rightarrow\mathbb{R}$ be two radial, real valued multipliers, and let
\begin{equation}\label{tetaw}
\Theta_W(t)=\int_{\mathbb{R}^n}\left(\phi|u_t|^2+\phi a(\nabla_bu,\nabla_bu)-\frac{1}{2}(A\phi)|u|^2\right)dx
+\int_{\mathbb{R}^n}|u|^2(V\phi+\psi)dx.
\end{equation}
Then for the solution $u$ of \eqref{eq2} with initial data $f$, $g\in L^2$, $Hf$, $Hg\in L^2$ the following virial-type identity holds:
\begin{equation}\label{virw}
\ddot{\Theta}_W(t)=2\int_{\mathbb{R}^n}\nabla_buD^2_a\phi\overline{\nabla_bu}dx-\frac{1}{2}\int_{\mathbb{R}^n}|u|^2A^2\phi dx+
\end{equation}
$$
+
2\int_{\mathbb{R}^n}|u_t|^2\psi dx-2\int_{\mathbb{R}^n}a(\nabla_bu,\nabla_bu)\psi dx+\int_{\mathbb{R}^n}|u|^2A\psi dx
$$
$$
+\int_{\mathbb{R}^n}(2\psi V-\phi' V^a_r)|u|^2 dx
+2\mathcal{I}\int_{\mathbb{R}^n}u\phi'a(\nabla_b u,B_\tau^a)dx
$$
$$
+2\int_{\mathbb{R}^n}\big[2a\left(\nabla_bu,\nabla a(\nabla\phi,\nabla_bu)\right)-
a\left(\nabla\phi,\nabla a(\nabla_bu,\nabla_bu)\right)\big]dx.
$$
\end{theorem}

\begin{remark}
We point out that, to the best of our comprehension, the term $\int2\psi V|u|^2$ that appears here 
was overlooked in the analogous computation
in \cite{fanvega}. 
\end{remark}

As a standard application of the above virial identities 
(see \cite{fanvega},\cite{bousdan}) we prove smoothing estimates for the Schr\"odinger equation in dimension 3 in a perturbative setting. Similar estimates can be proved for the wave equation with minor modifications, but we prefair not to pursue this topic here. 
The assumptions on the potentials will be expressed 
as usual
in terms of the Morrey-Campanato norms \begin{equation}\label{norm1}
\|| f\|_{\mathcal{C}^\alpha}=\int_0^\infty\rho^\alpha\sup_{|x|=\rho}|f(x)|\:d\rho.
\end{equation}

\begin{theorem}\label{smooth1}
Let $n=3$, assume the matrix $[a_{jk}]$ has the form
\begin{equation*}
a_{jk}=\delta_{jk}+\varepsilon\tilde{a}_{jk},
\end{equation*}
 satisfies \eqref{strutta}
and
the following assumption
\begin{equation}\label{hppert}
\sup_{j,k}\left(\sup_{|\gamma|=p}|\partial^{\gamma}\tilde{a}_{jk}|\right)\lesssim 
\frac{C}{\langle x\rangle^{p+}},
\qquad p=0,\dots,3,\quad\gamma\in\mathbb{N}^n.
\end{equation}
Assume moreover that the potentials are such that
\begin{equation}\label{hppot}
\||B_\tau^a|^2\|_{\mathcal{C}^3}+\|(V_r^a)_+\|_{\mathcal{C}^2}\leq\frac{1-\epsilon}2,\qquad |V(x)|\lesssim \frac{c_\varepsilon}{|x|^2}
\end{equation}
for some $c_\varepsilon$ sufficiently small.
Then the solution $u$ of \eqref{eq1} corresponding to
 $f\in L^2$, $Hf\in L^2$ satisfies the estimate
\begin{equation}\label{smoothsc}
\sup_{R>0}\frac{1-C_\varepsilon}{R}\int_0^{+\infty}\int_{|x|\leq R}|\nabla_bu|^2dxdt\lesssim \|f\|^2_{\mathcal{\dot{H}}^{\frac12}}.
\end{equation}
for some constant $C_\varepsilon<1$.
\end{theorem}

\begin{remark}
The second assumption of smallness on the potential in \eqref{hppot}, not necessary in the unperturbed case, seems to be removable at the cost of some additional technicalities.
\end{remark}

\begin{remark}
Hypothesis  \eqref{hppert} is naturally satisfied if for instance $$\tilde{a}_{jk}=\langle x\rangle^{-p}\delta_{jk},$$ with $p>0$.
\end{remark}

\begin{remark}
Following the strategy of \cite{fanvega}, it is possible to obtain smoothing estimates for both the Schr\"odinger and wave equation in higher dimensions $n\geq4$ using a suitable modifications of the multiplier used in the proofs of the previous theorems, under slightly different assumptions on the potential. This problem will be pursued in forthcoming papers.
\end{remark}

\begin{remark}\label{ree}
The non perturbative case remains an interesting open problem: we do not know whether virial identities \eqref{virsch} and \eqref{virw} could be used to prove smoothing estimates in cases when the matrix $a$ is not close to the identity. Even the radial case, in which the choice of the multiplier seems to be forced to the standard one, seems to require some new tools. Furthermore, it would be very interesting to try to adapt the choice of  multiplier to the scalar product $a$, regarding it as a new metric (see for instance \cite{staff} in the case of the hyperbolic space). We recall also that very few results exist when the coefficients depend also on time (and then the estimates obtained are only local). It is possible to adapt the above techniques to study this case, at least in special situations, including some equations with nonlocal nonlinearities (see 
e.g.\cite{DanconaSpagnolo95-a}, \cite{DanconaSpagnolo92-a}). 
These topics will be the object of future work.
\end{remark}

The paper is organized as follows: in Section \ref{virials} and \ref{smooths} 
we prove respectively Theorems \ref{virial}-\ref{virial2} and \ref{smooth1}, while the final Section is devoted to the proof of a quite general weighted magnetic Hardy's inequality that is needed in several steps of the paper.

\section{Proofs of virial identities}\label{virials}

\subsection{Schr\"odinger equation: proof of Theorem \ref{virial}}

Following the standard well-known strategy, we can easily compute for a solution $u\in\mathcal{H}^\frac{3}{2}$ of \eqref{eq1} 
\begin{equation}\label{der1}
\dot{\Theta}_S(t)=-i\langle u,[H,\phi]u\rangle
\end{equation}
\begin{equation}\label{der2}
\ddot{\Theta}_S(t)=-\langle u,[H,[H,\phi]]u\rangle
\end{equation}
where $[\cdot,\cdot]$ is the standard commutator and $\langle\cdot,\cdot\rangle$ is the hermitian product in $L^2$. Denoting with 
$$
T=-[H,\phi]
$$
it is trivial to compute explicitly, by Leibnitz formula, 
\begin{equation}\label{T}
T=A\phi+2a(\nabla\phi,\nabla_b).
\end{equation}
Hence we can rewrite \eqref{der2} as
\begin{equation}\label{der2bis}
\ddot{\Theta}_S(t)=\langle u,[H,T]u\rangle
\end{equation}
with $T$ given by \eqref{T}. We thus need to compute the commutator $[H,T]$; we divide it into three parts as follows
\begin{equation}\label{comht}
[H,T]u=[H_0,2a(\nabla\phi,\nabla_b)]u+[H_0,A^2\psi]u+[V,T]u=-I-II+III.
\end{equation}
The term $III$ is trivial: indeed we immediately have
\begin{equation}\label{iii}
[V,T]u=2[V,a(\nabla\phi,\nabla_b)]u=-2a(\nabla\phi,\nabla V)u=-2\phi'V^a_ru
\end{equation}
with $V^a_r=a(\hat{x},\nabla V)$, so that the corresponding term in \eqref{der2}, obtained integrating \eqref{ii} multiplied by $\overline{u}$ gives
\begin{equation}\label{primo}
-2\int_{\mathbb{R}^n} \phi'V^a_r|u|^2dx.
\end{equation}
We now turn to $II$. Writing the components in details we have
\begin{equation}\label{ii}
[H_0,A^2\phi]u=
\end{equation}
$$
-\partial^b_j\left[a_{jk}\partial_k^b(\partial_l(a_{lm}\partial_m\phi) u)\right]+
\partial_l(a_{lm}\partial_m\phi)\partial^b_j(a_{jk}\partial^b_k u)=-II_1+II_2.
$$
For the term $II_1$ we have
\begin{equation}\label{a1}
II_1=\partial^b_j\left[a_{jk}\partial_kA\phi u+a_{jk}A\phi\partial^b_k u\right]=
\end{equation}
$$
=A^2\phi\cdot u+2a_{jk}\partial_k (A\phi)\partial^b_ju+A\phi(\partial_ja_{jk})\partial^b_k u+(A\phi) a_{jk}\partial^b_j\partial^b_k u.
$$
Term $II_2$ reads instead as
\begin{equation}\label{b1}
II_2=A\phi(\partial_ja_{jk})\partial^b_ku+(A\phi) a_{jk}\partial^b_j\partial^b_k u.
\end{equation}
Hence plugging \eqref{a1} and \eqref{b1} into \eqref{ii} yields
\begin{equation}\label{iibis}
[H_0,A^2\phi]u=-A^2\phi\cdot u-2a_{jk}\partial_k (A\phi)\partial^b_j u.
\end{equation}
Finally we handle the term $I$. We have
\begin{equation}\label{i}
[H_0,2a(\nabla_b,\nabla\phi)]u=
\end{equation}
$$
-2\partial_j^b\left[a_{jk}\partial^b_k(a_{lm}\partial_l\phi\partial^b_mu)\right]+
2a_{lm}\partial_l\phi\partial^b_m[\partial_j^b(a_{jk}\partial_k^bu)]=-2(I_1-I_2).
$$
We treat separately the two terms. Starting with $I_1$ we have
$$
I_1=\partial_j^b\cdot\{a_{jk}(\partial_ka_{lm})\partial_l\phi\partial^b_m u+a_{jk}a_{lm}\partial_k\partial_l\phi\partial_m^b u+
a_{jk}a_{lm}\partial_l\phi\partial^b_k\partial^b_mu\}.
$$
We now turn to the last term $I_2$, that is the most tricky to handle. We need the following remark
\begin{equation}\label{scambio}
\partial^b_m\partial^b_j=\partial^b_j\partial^b_m+i(\partial_mb_j-\partial_jb_m)
\end{equation}
to write
$$
I_2=a_{lm}\partial_l\phi\partial^b_j[\partial_m^b(a_{jk}\partial_k^bu)]+ia_{lm}\partial_l\phi(\partial_mb_j-\partial_jb_m)a_{jk}\partial^b_ku.
$$
We expand the first term beyond again to have
$$
a_{lm}\partial_l\phi\partial^b_j[\partial_m^b(a_{jk}\partial_k^bu)]
$$
$$
=\partial_j^b\cdot\left\{a_{lm}\partial_l\phi(\partial_ma_{jk})\partial_k^bu+a_{lm}\partial_l\phi a_{jk}\partial^b_m\partial_k^bu\right\}-\partial_j(a_{lm}\partial_l\phi)\partial^b_m(a_{jk}\partial^b_ku)
$$
$$
=\partial_j^b\cdot\left\{a_{lm}\partial_l\phi(\partial_ma_{jk})\partial_k^bu+a_{lm}\partial_l\phi a_{jk}\partial^b_m\partial_k^bu\right\}
$$
$$
-\partial^b_m\cdot\{(\partial_ja_{lm})\partial_l\phi a_{jk}\partial^b_k u+a_{lm}\partial_j\partial_l\phi a_{jk}\partial^b_k u\}+
\partial_j (\partial_m(a_{lm}\partial_l\phi)) a_{jk}\partial_k^b u.
$$
Now we integrate  $I+II$ over $\mathbb{R}^n$ and use the integration by parts to rearrange the terms in a suitable way, as follows. 
First of all using again \eqref{scambio} we have
\begin{equation}\label{uno} 
2\int_{\mathbb{R}^n}\Big[a_{jk}a_{lm}\partial_l\phi\partial^b_m\partial^b_ku-
a_{jk}a_{lm}\partial_l\phi\partial^b_k\partial^b_mu)\cdot\overline{\partial_j^bu}-ia_{lm}\partial_l\phi(\partial_mb_j-\partial_jb_m)a_{jk}\partial^b_ku\cdot\Big]\overline{u}dx
\end{equation}
$$
=4\mathcal{I}\int_{\mathbb{R}^n}\Big(a_{jk}a_{lm}\partial_l\phi(\partial_kb_m-\partial_mb_k)\overline{\partial^b_ju}\Big)\cdot u dx
$$
$$
=4\mathcal{I}\int_{\mathbb{R}^n}u\phi'a(\nabla_b u,B_\tau^a)dx
$$
where we have defined $B^a_\tau=a(\hat{x},B)$.

Two other terms give
\begin{equation}\label{due}
2\int_{\mathbb{R}^n}\Big(\partial_j^b\cdot\left\{a_{jk}a_{lm}\partial_k\partial_l\phi\partial_m^bu\right\}+
\partial^b_m\cdot\left\{a_{jk}a_{lm}\partial_j\partial_l\phi\partial^b_k u\right\}\Big)\overline{u}\:dx
\end{equation}
$$
=4\int_{\mathbb{R}^n}\nabla_buD^2_a\phi\overline{\nabla_bu}\:dx
$$
where we have defined the distorted Jacobian $D^2_a\phi=a_{jk}a_{lm}\partial_k\partial_l\phi$.
\\The remaining terms are given by
\begin{equation}\label{tre}
2\int_{\mathbb{R}^n}\Big[\partial^b_j\cdot\{a_{jk}(\partial_ka_{lm})\partial_l\phi\partial^b_m u-a_{lm}(\partial_ma_{jk})\partial_l\phi\partial_k^bu\}+
\partial^b_m\cdot\{a_{jk}(\partial_ja_{lm})\partial_l\phi\partial_k^bu\}\Big]\overline{u} \:dx.
\end{equation}
Denoting with $\displaystyle\nabla a(w,x)=\sum_l\partial_l a_{jk}(\overline{w}_jz_k)$ we can rewrite \eqref{tre} as 
\begin{equation}\label{curv!}
= 2\int_{\mathbb{R}^n}\big[2a\left(\nabla_bu,\nabla a(\nabla\phi,\nabla_bu)\right)-
a\left(\nabla\phi,\nabla a(\nabla_bu,\nabla_bu)\right)\big] dx.
\end{equation}
Thus putting all together, \eqref{primo}, \eqref{uno}, \eqref{due}, \eqref{tre} and the integration of \eqref{iibis} gives \eqref{virsch}.

\subsection{Wave equation: proof of Theorem \ref{virial2}}

The proof of Theorem \ref{virial2} is analogous to the previous one. We consider the following quantity
\begin{equation}\label{tetaw}
\Theta_W(t)=\int_{\mathbb{R}^n}\left(\phi|u_t|^2+\phi a(\nabla_bu,\nabla_bu)-\frac{1}{2}(A\phi)|u|^2\right)dx+
\int_{\mathbb{R}^n}\phi V|u|^2dx+\int_{\mathbb{R}^n}|u|^2\psi \:dx.
\end{equation}
Differentiating \eqref{tetaw} in time, integrating by parts and using equation \eqref{eq2} yield, term by term, 
$$
\frac{d}{dt}\int\phi|u_t|^2dx=-2\mathcal{R}\langle u_t,\phi H_0u\rangle-
2\mathcal{R}\langle u_t,\phi Vu\rangle;
$$
$$
\frac{d}{dt}\int_{\mathbb{R}^n}\phi a(\nabla_bu,\nabla_bu)dx=
+2\mathcal{R}\langle u_t,\phi H_0u\rangle-
2\mathcal{R}\langle u_t,a(\nabla\phi,\nabla_bu)\rangle;
$$
$$
-\frac{1}{2}\frac{d}{dt}\int_{\mathbb{R}^n}(A\phi)|u|^2dx=-\mathcal{R}\langle u_t,(A\phi)u\rangle;
$$
$$
\frac{d}{dt}\int_{\mathbb{R}^n}\phi V|u|^2dx=
2\mathcal{R}\langle u_t,\phi Vu\rangle;
$$
$$
\frac{d}{dt}\int_{\mathbb{R}^n}|u|^2\psi \:dx=2\mathcal{R}\langle u_t,\psi u\rangle.
$$
Recalling the definition of 
$$
T=-[H,\phi]=A\phi+2a(\nabla\phi,\nabla_b)
$$
we thus immediately have 
\begin{equation}\label{tetaw'}
\dot{\Theta}_W(t)=-\mathcal{R}\langle u_t,Tu\rangle+2\mathcal{R}\langle u_t,\psi u\rangle.
\end{equation}
Differentiating the first term on the RHS of \eqref{tetaw} and using equation \eqref{eq2} yields
\begin{equation}\label{ap1}
-\frac{d}{dt}\langle u_t,Tu\rangle=\langle u,HT u\rangle-\langle u_t,Tu_t\rangle.
\end{equation}
Notice that the antisymmetry of $T$ yields 
\begin{equation}\label{ap2}
\mathcal{R}\langle u_t,Tu_t\rangle=0;
\end{equation} 
moreover we have
$$
\langle u,HTu\rangle=\langle u,tHu\rangle+\langle u,[H,T]u\rangle=
-\langle HTu,u\rangle+\langle u,[H,T]u\rangle
$$
and hence
\begin{equation}\label{ap3}
\mathcal{R}\langle u,HTu\rangle=\frac{1}{2}\langle u,[H,T]u\rangle.
\end{equation}
Plugging \eqref{ap2}, \eqref{ap3} into \eqref{ap1} yields
\begin{equation}\label{ap4}
\frac{d}{dt}\mathcal{R}\langle u_t,Tu\rangle=\frac{1}{2}\langle u,[H,T]u\rangle.
\end{equation}
We now turn to the derivative of the second term on the RHS of \eqref{tetaw'} to write
\begin{equation}\label{ap5}
2\frac{d}{dt}\mathcal{R}\langle u_t,\psi u\rangle=2\langle u_t,\psi u_t\rangle+2\mathcal{R}\langle H u,\psi u\rangle.
\end{equation}
Integrating by parts immediately yields
\begin{equation}\label{ap6}
\mathcal{R}\langle H u,\psi u\rangle=-\int a(\nabla_bu,\nabla_bu)\psi \:dx-\mathcal{R}\int a(\nabla_b u,\nabla\psi)\overline{u}\:dx+\int \psi V |u|^2 dx.
\end{equation}
Note now that
$$
\int a(\nabla_b u,\nabla\psi)\overline{u}\:dx=\int|u|^2A\psi dx-\mathcal{R}\int a(\nabla\psi,\nabla_b u)udx
$$
so that
\begin{equation}\label{ap7}
\mathcal{R}\int a(\nabla_b u,\nabla\psi)\overline{u}\:dx=-\frac{1}{2}\int|u|^2A\psi\: dx.
\end{equation}
Hence by \eqref{ap5}, \eqref{ap6} and \eqref{ap7} we obtain
\begin{equation}\label{ap8}
2\frac{d}{dt}\mathcal{R}\langle u_t,\psi u\rangle=
2\int|u_t|^2\psi\: dx-2\int a(\nabla_bu,\nabla_bu)\psi\: dx+\int|u|^2A\psi\: dx+2\int\psi V|u|^2dx.
\end{equation}
In conclusion, by \eqref{tetaw'}, \eqref{ap4} and \eqref{ap8} we have obtained the following identity
\begin{eqnarray}\label{ap9}
\ddot{\Theta}_W(t) &=& \frac{1}{2}\langle u,[H,T]u\rangle+2\int|u_t|^2\psi\: dx-2\int a(\nabla_bu,\nabla_bu)\psi \:dx
\\\nonumber
&+&
\int|u|^2A\psi\: dx+2\int\psi V|u|^2dx
\end{eqnarray}
which, together with the explicit calculation of the term $\langle u,[H,T]u\rangle$ already discussed in the proof of Theorem \ref{virial}, yields \eqref{virw}.

\section{Proofs of the smoothing estimates}\label{smooths}

The proof of Theorem \ref{smooth1} is based on the standard multiplier method, and in particular it folllows closely the strategy of \cite{fanvega}; the main tool consists in the choice of a suitable radial function $\phi$ to be put in virial identity \eqref{virsch} and the explicit evaluation of the integrals that come into play. The handling of the generalized bilaplacian term will be simplified by the introduction of a new easy identity involving a further small multiplier $\varphi$.

First of all, using integration by parts and relation \eqref{der1} we notice that
$$
\dot{\Theta}_S(t)=2\mathcal{I}\int_{\mathbb{R}^n}a(\nabla\phi,\nabla_bu)\overline{u}\:dx,
$$
so that we can rewrite virial identity \eqref{virsch} as
\begin{equation}\label{virsch2}
2\int_{\mathbb{R}^n}\nabla_buD^2_a\phi\overline{\nabla_bu}\:dx-\frac{1}{2}\int_{\mathbb{R}^n}|u|^2A^2\phi\: dx-
\int_{\mathbb{R}^n}\phi'V_r^a|u|^2dx+
\end{equation}
$$
+2\mathcal{I}\int_{\mathbb{R}^n}u\phi'a(B_\tau^a,\nabla_b u)dx
$$
$$
+\int_{\mathbb{R}^n}\big[2a\left(\nabla_bu,\nabla a(\nabla\phi,\nabla_bu)\right)-
a\left(\nabla\phi,\nabla a(\nabla_bu,\nabla_bu)\right)\big]dx=K(t)
$$
for all $t>0$, where we have defined the quantity
$$
K(t)=\frac{d}{dt}\left(\mathcal{I}\int_{\mathbb{R}^n}a(\nabla_bu,\nabla\phi)\overline{u} \:dx\right).
$$
We thus need the following interpolation Lemma, valid for every dimension $n\geq3$, to control this term.

\begin{lemma}\label{interpol}
Let $n\geq3$, $[a_{jk}(x)]$ be a symmetric real matrix of smooth functions satisfying \eqref{strutta}, and $\phi:\mathbb{R}^n\rightarrow\mathbb{R}$ be a radial function such that $\phi'$ is bounded and
\begin{equation}\label{hpA}
A\phi=\partial_j(a_{jk}\partial_k \phi)\lesssim \frac{1}{|x|}.
\end{equation}
 Then the following esimate holds:
\begin{equation}\label{rhs}
\left|\int_{\mathbb{R}^n}a(\nabla_bu,\nabla\phi)\overline{u}dx\right|\leq\|f\|_{{\mathcal{\dot H}^\frac{1}{2}}}^2
\end{equation}
\end{lemma}

\begin{remark}\label{rik}
Condition \eqref{hpA} is ensured if for instance $\phi'(r)$, $r\phi''(r)$ are bounded and hypothesis \eqref{hppert} holds, or as well if the matrix $[a_{jk}(x)]$ is of the form $h(|x|){\rm Id}$ and the functions $h(|x|)$ and $|x|h'(|x|)$ are bounded.
\end{remark}

\begin{proof}
Consider the quadratic form
$$
T(f,g)=\int_{\mathbb{R}^n}\overline{f}\:a_{jk}\partial^b_jg\partial_k\phi\: dx=
\int_{\mathbb{R}^n}\overline{f}\:a_{jk}^\frac12\partial^b_jg\cdot a_{jk}^\frac12\partial_k\phi \:dx
$$
which is well defined since the matrix $[a_{jk}(x)]$ is symmetric and positive definite. 
From the boundedness of $\phi'$ and $a$ we immediately have the first estimate, by H\"older's inequality,
\begin{equation}\label{inter1}
\displaystyle |T(f,g)|
\lesssim \|f\|_{L^2}\|a_{jk}^\frac12\partial^b_jg\|_{L^2}=
\|f\|_{L^2}\|g\|_{\mathcal{\dot{H}}^1}.
\end{equation}
Now we integrate by parts to have
\begin{equation*}
T(f,g)=-\int g\:a_{jk}^\frac12\partial^b_j\overline{f}\cdot a_{jk}^\frac12\partial_k\phi\:dx-\int g \overline{f}\:A\phi\:dx.
\end{equation*}
Using hypothesis \eqref{hpA}, H\"older and Magnetic Hardy's inequality (cfr. Proposition \ref{append}) we have
\begin{equation}\label{inter2}
\displaystyle |T(f,g)|\leq\| \sup_{j,k}(a_{jk}^{1/2})g\|_{L^2}\cdot\|a_{jk}^{1/2}\partial^b_j f\|_{L^2}
\end{equation}
$$
\lesssim \|g\|_{L^2}\|f\|_{\mathcal{\dot{H}}^1}.
$$
Interpolation between \eqref{inter1} and \eqref{inter2} yields
\begin{equation*}
\displaystyle |T(f,g)|\lesssim \|f\|_{\mathcal{\dot{H}}^\frac12} \|g\|_{\mathcal{\dot{H}}^\frac12}
\end{equation*}
and thus the conservation of the $\dot{\mathcal{H}}^s$ norm concludes the proof.

\end{proof}

Now we turn to the RHS of \eqref{virsch}. Before fixing the multipliers, we state an easy lemma that will be very helpful to us.
\begin{lemma}
For every $u$ solution of \eqref{eq1} and $\varphi:\mathbb{R}^n\rightarrow\mathbb{R}$ the following equality holds
\begin{equation}\label{inpart}
\int A\varphi|u|^2-\int\varphi a(\nabla_bu,\nabla_bu)-\int\varphi V|u|^2=0.
\end{equation}
\end{lemma}
\begin{proof}
Multiply equation \eqref{eq1} by $\phi\overline{u}$, integrate by parts and take the real part.
\end{proof}

Putting together \eqref{virsch} and \eqref{inpart} we thus obtain a new virial identity involving both $\phi$ and $\varphi$
\begin{equation}\label{virsch2}
\ddot{\Theta}_S(t)=4\int_{\mathbb{R}^n}\nabla_buD^2_a\phi\overline{\nabla_bu}dx-\int_{\mathbb{R}^n}|u|^2A(A\phi+\varphi) dx-
2\int_{\mathbb{R}^n}\phi'V_r^a|u|^2dx
\end{equation}
$$
+\int\varphi a(\nabla_bu,\nabla_bu)+\int\varphi V|u|^2+4\mathcal{I}\int_{\mathbb{R}^n}u\phi'a(\nabla_b u,B_\tau^a)dx
$$
$$
+2\int_{\mathbb{R}^n}\big[2a\left(\nabla_bu,\nabla a(\nabla\phi,\nabla_bu)\right)-
a\left(\nabla\phi,\nabla a(\nabla_bu,\nabla_bu)\right)\big]dx.
$$

Now we choose an explicit radial multiplier $\phi$, first. Let us consider, for some $M>0$ to be fixed later, the function

\begin{equation}\label{phi0}
\phi_0(r)=\int_0^r\phi'(s)ds
\end{equation}
where
$$
\phi'_0(r)=
\begin{cases}
M+\frac{1}{3}r,\quad r\leq1
\\
M+\frac{1}{2}-\frac{1}{6r^2},\quad r>1.
\end{cases}
$$
We then define the scaled function
$$
\phi(r)=\phi_R(r)=R\phi_0\left(\frac{r}{R}\right)
$$
for which we can easily compute
\begin{equation}\label{phi1}
\phi'(r)=
\begin{cases}
M+\frac{r}{3R},\quad r\leq R
\\
M+\frac{1}{2}-\frac{R^2}{6r^2},\quad r>R,
\end{cases}
\end{equation}
\begin{equation}\label{phi2}
\phi''(r)=
\begin{cases}
\frac{1}{3R},\quad r\leq R
\\
\frac{1}{R}\cdot\frac{R^3}{3r^3},\quad r>R,
\end{cases}
\end{equation}
\begin{equation}\label{lapphi}
\Delta\phi=
\begin{cases}
\frac{1}{R}+\frac{2M}{r},\quad r\leq R
\\
\frac{1+2M}{r},\quad r>R.
\end{cases}
\end{equation}
\begin{equation}\label{bilap}
\Delta^2\phi(r)=-4\pi M\delta_{x=0}-\frac{1}{R^2}\delta_{|x|=R}
\end{equation}
(notice that this function $\phi$ satisfies the hypothesis of Lemma \ref{interpol}). 

We pick instead the second multiplier to be $\varphi=-\varepsilon\tilde{A}\phi$. It is easy to verify that this function is well defined, continuous and such that
\begin{equation}\label{mult2}
|\varphi(x)|\leq\frac{C_\varepsilon}{\langle x\rangle^{1+}}
\end{equation}
The introduction of this new multiplier takes the advatange of making the term $A\phi+\varphi$ radial, and this will facilitate some steps of our proof.

We begin with the following important relation, valid in any dimension.
\begin{equation}\label{rel}
\nabla _bu D^2_a\phi\overline{\nabla_bu}=\phi''\left|\nabla^{a,r}_bu\right|^2+\frac{\phi'}{r}\left|\nabla^{a,\tau}_bu\right|^2
\end{equation}
where $\nabla^{a,r}_b$ and $\nabla^{a,\tau}_b$ denote respectively the radial and tangential component of the covariant gradient with respect to the hermitian product $a(\cdot,\cdot)$, so that 
\begin{equation}\label{rel1}
|\nabla^{a,r}_bu|^2=\displaystyle a\left(\hat{x},\nabla_bu\right)^2
\end{equation}
and
\begin{equation}\label{rel2}
\quad\left|\nabla^{a,\tau}_bu\right|^2=\left|a_{jk}\cdot\partial_j^bu\right|^2-a\left(\hat{x},\nabla_bu\right)^2.
\end{equation}
Notice moreover that our choice of $\varphi$ yields
\begin{equation*}
A(A\phi+\varphi)=A(\Delta\phi).
\end{equation*}
Thus plugging \eqref{bilap} and  \eqref{rel} into \eqref{virsch2} and neglecting the negative part $V^{a-}_r $ of the electric potential yield
\begin{equation}\label{virsch3}
\frac{2}{3R}\int_{|x|\leq R}\left|a_{jk}\cdot\partial_j^bu\right|^2dx+2M\int_{\mathbb{R}^n}\frac{|\nabla^{a,\tau}_bu|^2}{|x|}dx
-\frac{1}{2}\int_{\mathbb{R}^n}|u|^2A(\Delta\phi)dx+
\end{equation}
$$
-\int_{\mathbb{R}^n}\phi'V^{a+}_r|u|^2dx+2\mathcal{I}\int_{\mathbb{R}^n}
u\phi'a(\nabla_bu,B^a_\tau)dx
$$
$$
+\int_{\mathbb{R}^n}\varphi a(\nabla_bu,\nabla_bu)dx+\int_{\mathbb{R}^n}\varphi V|u|^2dx$$
$$
+\int_{\mathbb{R}^n}\big[2a\left(\nabla_bu,\nabla a(\nabla\phi,\nabla_bu)\right)
a\left(\nabla\phi,\nabla a(\nabla_bu,\nabla_bu)\right)\big]dx
$$
$$
\leq \frac{d}{dt}\left(\mathcal{I}\int_{\mathbb{R}^n}a(\nabla_bu,\nabla\phi,)\overline{u}\:dx\right)
$$
for all $R>0$.

Now we estimate the LHS of \eqref{virsch3} term by term. First of all notice that
\begin{eqnarray}\label{banal}
\int_{\mathbb{R}^n}\varphi a(\nabla_bu,\nabla_bu)dx&\geq&
-C_\varepsilon\left|\int_{\mathbb{R}^n}\frac{a(\nabla_bu,\nabla_bu)}{\langle x\rangle^{1+}}dx\right|
\\
\nonumber
&\geq&
-C_\varepsilon\sup_{R>0}\frac1R\int_{|x|\leq R}|\nabla_bu|^2dx
\end{eqnarray}
The magnetic potential term then gives
\begin{equation}\label{B}
2\mathcal{I}\int_{\mathbb{R}^n}u\phi'a(\nabla_bu,B^a_\tau)dx \geq
-2\left|\mathcal{I}\int_{\mathbb{R}^n}u\phi'a(\nabla_bu,B^a_\tau)dx\right|
\end{equation}
$$
\geq -2\left(M+\frac{1}{2}\right)\int_{\mathbb{R}^n}|u|\cdot|B^a_\tau|\cdot|\nabla^{a,\tau}_bu|dx
$$
$$
\geq-2\left(M+\frac{1}{2}\right)\left(\int\frac{\left|\nabla^{a,\tau}_bu\right|^2}{|x|}dx\right)^\frac{1}{2}\left(\int_0^\infty d\rho\int_{|x|=\rho}|x|\cdot|u|^2\cdot|B_\tau^a|^2d\sigma\right)^\frac{1}{2}
$$
$$
\geq-2\left(M+\frac{1}{2}\right)\left(\int\frac{\left|\nabla^{a,\tau}_bu\right|^2}{|x|}dx\right)^\frac{1}{2}\left(\sup_{R>0}\frac{1}{R^2}\int_{|x|=R}|u|^2d\sigma\right)^\frac{1}{2}\|(B^a_\tau)^2\|_{\mathcal{C}^3}^\frac12.
$$
For the electric potential terms we have, in a similar way,
\begin{equation}\label{V}
-\int_{\mathbb{R}^n}\phi'(V_r^a)_+|u|^2dx\geq-\left|\int_{\mathbb{R}^n}\phi'(V_r^a)_+|u|^2dx\right|
\end{equation}
$$
\geq-\left(M+\frac{1}{2}\right)\int_0^{+\infty}d\rho\int_{|x|=\rho}|(V_r^a)_+|\cdot|u|^2d\sigma
$$
$$
\geq-\left(M+\frac{1}{2}\right)\int_0^{+\infty}d\rho\left(\sup_{|x|=\rho}(|(V_r^a)_+|\cdot|x|^2)\frac{1}{\rho^2}\int_{|x|=\rho}|u|^2d\sigma\right)
$$
$$
\geq-\left(M+\frac{1}{2}\right)\left(\sup_{R>0}\frac{1}{R^2}\int_{|x|=R}|u|^2d\sigma\right)
\|(V_r^a)_+\|_{\mathcal{C}^2}
$$
and, using proposition \ref{append}
\begin{eqnarray}\label{V2}
\int_{\mathbb{R}^n}\varphi V|u|^2dx&\geq&-C_\varepsilon\left|\int_{\mathbb{R}^n}V\frac{|u|^2}{\langle x\rangle^{1+}}dx\right|
\\
\nonumber
&\geq&
-C_\varepsilon\left|\int_{\mathbb{R}^n}V\frac{|u|^2}{\langle x\rangle^{1+}}dx\right|
\\
\nonumber
&\geq&
-C_\varepsilon(\|V|x|^2\|_{L^\infty})\left|\int_{\mathbb{R}^n}\frac{|u|^2}{|x|^2\langle x\rangle^{1+}}dx\right|
\\
\nonumber
&\geq&
-C_\varepsilon(\|V|x|^2\|_{L^\infty})\int_{\mathbb{R}^n}\frac{|\nabla_bu|^2}{\langle x\rangle^{1+}}dx
\\
\nonumber
&\geq&
-C_\varepsilon(\|V|x|^2\|_{L^\infty})\sup_{R>0}\frac1R\int_{|x|\leq R}|\nabla_bu|^2dx.
\end{eqnarray}
Now we focus on the term
\begin{equation*}
2\int_{\mathbb{R}^n}\big[2a\left(\nabla_bu,\nabla a(\nabla\phi,\nabla_bu)\right)-
a\left(\nabla\phi,\nabla a(\nabla_bu,\nabla_bu)\right)\big]dx.
\end{equation*}
Since the matrix $[a_{jk}(x)]$ is a perturbation of the identity we can rewrite as
\begin{equation*}
2\varepsilon\int_{\mathbb{R}^n}\big[2a\left(\nabla_bu,\nabla\tilde a(\nabla\phi,\nabla_bu)\right)-
a\left(\nabla\phi,\nabla \tilde a(\nabla_bu,\nabla_bu)\right)\big]dx
\end{equation*}
and estimate it, using  \eqref{strutta}, \eqref{hppert} and proposition \ref{append}, with
\begin{equation}\label{curvone}
\geq -2\varepsilon\int_{\mathbb{R}^n}2\left|\frac{(\nabla\phi\cdot\nabla_bu)(\hat{x}\cdot\overline{\nabla_bu})}{\langle x\rangle^{1+}}\right|+
\left|\frac{|\nabla_bu|^2(\nabla\phi\cdot\hat{x})}{\langle x\rangle^{1+}}\right|dx
\end{equation}

\begin{equation*}
\geq-2\varepsilon\int_{\mathbb{R}^n}\frac{\phi'}{\langle x\rangle^{1+}} \big[2|\nabla_b^ru|^2+ |\nabla_bu|^2\big]dx
\end{equation*}
\begin{equation*}
\geq-2\varepsilon\left(M+\frac12\right)\int_{\mathbb{R}^n}\frac{3|\nabla_bu|^2}{\langle x\rangle^{1+}} dx
\end{equation*}
\begin{equation*}
\geq-C_\varepsilon\sup_{R>0}\frac1R\int_{|x|\leq R}|\nabla_bu|^2dx.
\end{equation*}

Finally we turn to the generalized bilaplacian term. Due to \eqref{hppert}, \eqref{bilap} and the choice of $\varphi$ it is easily seen that
\begin{equation}\label{bilterm}
-\frac{1}{2}\int_{\mathbb{R}^n}|u|^2A(\Delta\phi)dx
=(2\pi M-\varepsilon)|u(0)|^2+\frac{1-\varepsilon}{2R^2}\int_{|x|=R}|u|^2d\sigma
\end{equation}
for some small constant $c_\varepsilon$.
Before going on we introduce the following compact notation, for the seek of simplicity:
$$
C_1:=\left(\int\frac{|\nabla^{a,\tau}_bu|^2}{|x|}dx\right)^\frac{1}{2};\qquad
C_2:=\left(\sup_{R>0}\frac{1}{R^2}\int_{|x|=R}|u|^2d\sigma\right)^\frac{1}{2}.
$$
Plugging now \eqref{B}, \eqref{V}, \eqref{curvone} and \eqref{bilterm} into \eqref{virsch3}, neglecting some positive terms and taking the supremum over $R>0$ yields
$$
\sup_{R>0}\frac{2-C_\varepsilon}{3R}\int_{|x|\leq R}|\nabla_bu|^2dx+2MC_1^2+\frac{1-\varepsilon}{2}C_2^2
$$
$$
\leq 
K(t)+\left(M+\frac{1}{2}\right)\cdot\left(2C_1C_2\|(B^a_\tau)^2\|_{\mathcal{C}^3}^\frac12+C_2^2\|V^{a+}_r\|_{\mathcal{C}^2}\right)
$$
that is, equivalently,
\begin{equation}\label{zio}
\sup_{R>0}\frac{2-C_\varepsilon}{3R}\int_{|x|\leq R}|\nabla_bu|^2dx+C(C_1,C_2,M,B,V,a)
\end{equation}
$$
\leq K(t)
$$
where the constant $C$ is given by
\begin{equation*}
C(C_1,C_2,M,B,V,a)=2MC_1^2+\left[\frac{1-\varepsilon}2-\left(M+\frac12\right)\|V^{a+}_r\|_{\mathcal{C}^2}\right]C_2^2
\end{equation*}
$$
-2\left(M+\frac12\right)\|(B^a_\tau)^2\|_{\mathcal{C}^3}^\frac12 C_1C_2.
$$
To conclude the proof we need to optimize the smallness condition on $\|(B^a_\tau)^2\|_{\mathcal{C}^3}^\frac12$ and $\|V^{a+}_r\|_{\mathcal{C}^2}$ under which we can ensure that $C(C_1,C_2,M,B,V,a)\geq0$. We can fix $C_1=1$ since $C$ is homogeneous, and impose
\begin{equation*}
\left[\frac{1-\varepsilon}2-\left(M+\frac12\right)\|V^{a+}_r\|_{\mathcal{C}^2}\right]C_2^2
-2\left(M+\frac12\right)\|(B^a_\tau)^2\|_{\mathcal{C}^3}^\frac12C_2+2M>0
\end{equation*}
for all $C_2>0$. This gives the condition
\begin{equation}\label{eps}
\frac{(M+\frac12)^2}{M}\|(B^a_\tau)^2\|_{\mathcal{C}^3}+2\left(M+\frac12\right)\|V^{a+}_r\|_{\mathcal{C}^2}\leq1-\varepsilon.
\end{equation}
In view of minimizing the size of $B$ we choose $M=\frac12$; hence we obtain 
\begin{equation}\label{positive}
\|(B^a_\tau)^2\|_{\mathcal{C}^3}+\|V^{a+}_r\|_{\mathcal{C}^2}\leq\frac{1-\varepsilon}2\Rightarrow C(C_1,C_2,M,B,V,a)\geq0.
\end{equation}
Thus if \eqref{hppot} is satisfied from \eqref{zio} and \eqref{positive} we have\begin{equation*}
\sup_{R>0}\frac{1}{3R}\int_{|x|\leq R}|\nabla_bu|^2dx\leq CK(t)
\end{equation*}
for some positive constant $C$. 

The thesis comes from the integration in time of the last inequality, the application of Lemma \ref{interpol} and the conservation of the $\mathcal{\dot{H}}^\frac12$-norm at the right hand side.

\section{Hardy's inequality}

We here prove a quite general weighted magnetic Hardy's inequality that is needed in some parts of the paper.

\begin{proposition}[Weighted magnetic Hardy's inequality]\label{append}
Let $n\geq 3$, $A:\mathbb{R}^n\rightarrow\mathbb{R}^n$ and $w=w(|x|)$ a radial positive function such that
\begin{equation}\label{hpw}
|w(|x|)|\leq c_1,\qquad |w'(|x|)|\leq\frac{c_2}{|x|}
\end{equation}
for some constants $c_1$, $c_2$. Then the following inequality holds for any $f\in D(H)$ 
\begin{equation}\label{hardy2}
\int_{\mathbb{R}^n}\frac{|f|^2}{|x|^{2}}w(|x|)dx\leq C(n,c_1,c_2) \int_{\mathbb{R}^n}|\nabla_bf|^2w(|x|)dx.
\end{equation}
\end{proposition}

\begin{proof}
We write , for $\alpha\in\mathbb{R}$, 
\begin{equation}\label{dai}
0\leq\int_{\mathbb{R}^n}\left|(\nabla_bf) w(|x|)^\frac12+\alpha f\frac{x}{|x|^2}w(|x|)^\frac12\right|^2dx=
\end{equation}
\begin{equation*}
=\int_{\mathbb{R}^n}|\nabla_bf|^2w(|x|)dx+\alpha^2\int_{\mathbb{R}^n}\frac{|f|^2}{|x|^2}w(|x|)dx+
2\alpha\mathcal{R}\int_{\mathbb{R}^n}\overline{f}\frac{xw(|x|)}{|x|^2}\cdot\nabla_bfdx.
\end{equation*}
Integrating by parts the last term beyond becomes
\begin{equation*}
2\alpha\mathcal{R}\int_{\mathbb{R}^n}\overline{f}\frac{xw(|x|)}{|x|^2}\cdot\nabla_bfdx=
-\alpha\int_{\mathbb{R}^n}|f|^2{\rm div}\left(\frac{x}{|x|^2}w(|x|)\right)dx.
\end{equation*}
From \eqref{hpw} we have
\begin{equation}\label{divv}
-\left|{\rm div}\left(\frac{x}{|x|^2}w(|x|)\right)\right|=
-\left|\frac{(n-2)w(|x|)}{|x|^2}+\frac{w'(|x|)}{|x|}\right|
\end{equation}
\begin{equation*}
\geq -\frac{(n-2)c_1+c_2}{|x|^2},
\end{equation*}
that plugged into \eqref{dai} yields
\begin{equation}\label{bob}
\{-\alpha^2-((c_2+(n-2)c_1)\alpha\}\int_{\mathbb{R}^n}\frac{|f|^2}{|x|^{2}}w(|x|)dx\leq
\int_{\mathbb{R}^n}|\nabla_bf|^2w(|x|)dx
\end{equation}
that gives \eqref{hardy2}.

\end{proof}

\end{document}